\documentclass[11 pt, a4paper]{article}
\usepackage{amsmath}
\usepackage[ansinew]{inputenc}
\usepackage{amssymb}
\usepackage{amsthm}
\usepackage{amsfonts}
\usepackage{amscd}
\usepackage{latexsym, amsfonts}
\usepackage{color}
\usepackage{graphicx}

\title{On the length of arcs in labyrinth fractals}

\author{Ligia L. Cristea\thanks{L.L. Cristea is supported by the Austrian Science Fund (FWF), 
Project P27050-N26, by the Austrian Science Fund (FWF) Project F5508-N26, which is part of the Special Research Program ``Quasi-Monte Carlo Methods: Theory and Applications''. Part of this work was written while she was also supported by the Austrian-French cooperation project FWF I1136-N26} 
\\Karl-Franzens-Universit\"at Graz\\ Institut f\"ur Mathematik und Wissenschaftliches Rechnen
\\Heinrichstrasse 36, 8010 Graz,Austria\\ \tt{strublistea@gmail.com} 
 \and Gunther Leobacher\thanks{
G. Leobacher is supported by the Austrian Science Fund (FWF) Project F5508-N26, which is part of the Special Research Program ``Quasi-Monte Carlo Methods: Theory and Applications''. Part of this work was written when G. Leobacher worked at the Department of Financial Mathematics and Applied Number Theory at 
Johannes Kepler University Linz (JKU)
}\\Karl-Franzens-Universit\"at Graz\\ Institut f\"ur Mathematik und Wissenschaftliches Rechnen
\\Heinrichstrasse 36, 8010 Graz,Austria  \\ \tt{gunther.leobacher@uni-graz.at} }

\newtheorem{theorem}{Theorem}
\newtheorem{lemma}{Lemma}
\newtheorem{property}{Property}
\newtheorem{proposition}{Proposition}
\newtheorem{conjecture}{Conjecture}

\newcommand{\A}{\includegraphics{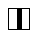}}
\newcommand{\B}{\includegraphics{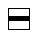}}
\newcommand{\C}{\includegraphics{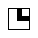}}
\newcommand{\D}{\includegraphics{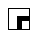}}
\newcommand{\E}{\includegraphics{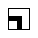}}
\newcommand{\F}{\includegraphics{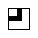}}

\definecolor{magenta}{rgb}{0.75,0,0.25}

\definecolor{violet}{rgb}{0.25,0,0.75}

\newcommand{\fr}{\mathrm{fr}}
\newcommand{\inter}{\mathrm{Int}}

\begin{document}

\maketitle

\abstract{Labyrinth fractals are self-similar dendrites in the unit square that are defined with the help of a labyrinth set or a labyrinth pattern. In the case when the fractal is generated by a horizontally and vertically blocked pattern, the arc between any two points in the fractal has infinite length \cite{laby_4x4, laby_oigemoan}.
In the case of mixed labyrinth fractals a sequence of labyrinth patterns is used in order to construct the dendrite. In the present article we focus on the length of the arcs between points of mixed labyrinth fractals.   
We show that, depending on the choice of the patterns in the sequence, both situations can occur:  the arc between any two points of the fractal has finite length, or the arc between any two points of the fractal has infinite length. This is in stark contrast to the self-similar case.}
\section{Introduction}\label{sec:introduction}

Labyrinth fractals are fractal dendrites in the plane, that can also be viewed as a special family of Sierpi\'nski carpets. Such carpets are not only studied by mathematicians, but also by physicists, e.g., as mathematical models for porous materials, rocks, or disordered media \cite{Tarafdar_modelporstructrepeatedSC2001, AnhHoffmanSeegerTarafdar2005}. The mathematical objects called \emph{labyrinth fractals} were introduced and studied by Cristea and Steinsky \cite{laby_4x4, laby_oigemoan, mixlaby}, on the one hand, and on the other hand in recent research in physics \cite{PotapovGrachev2012, GrachevPotapovGerman2013, PotapovGermanGrachev2013, PotapovZhang2016} objects called fractal labyrinths, strongly related to the labyrinth fractals, are used, as well as the labyrinth fractals mentioned above, including in \cite{GrachevPotapovGerman2013} the notation and mathematical frame introduced by Cristea and Steinsky \cite{laby_4x4, laby_oigemoan}. These fractal labyrinths and labyrinth fractals appear in physics in several different contexts. To the best of our knowledge, they first occured in the study of anomalous diffusion, and particle dynamics \cite{PotapovGrachev2012}. In \cite{GrachevPotapovGerman2013} they were used as a tool for processing and analysing planar nanostructures, while in \cite{PotapovGermanGrachev2013}the authors applied them in the context of fractal reconstruction of complex images, signals and radar backgrounds. In the very recent article \cite{PotapovZhang2016} it is shown how the benefits of the wide simulation abilities of labyrinth fractals were used in oder to create a software that generates  the shape of ultra-wide band fractal antennas, based on the geometry of labyrinth fractals, as introduced in \cite{laby_4x4}. Fractal antennas are already known to have applications, among others, in medicine, and cellular communications on base stations and mobile terminals, they have been of interest to scientists from the fields of physics and electronics for the last decade and still are a subject of ongoing research. 

In nature or in technics, objects that can be described by prefractals of labyrinth fractals occur in various situations: the system of blood or lymphatic vessels in the body of humans or animals, the leaf veins of plants, river systems, dendrites in the brain, the electrical discharches (e.g., lightening) on the one hand, and, on the other hand, systems of irrigation in agriculture, systems of ressources or information distribution, communication or transport networks. In the context of physics, a fractal labyrinth is defined \cite{PotapovGrachev2012} as ``a connected topological structure with fractal dimension greater than $1$ and with the scaling nature of the conducting channels''. 
Thus the labyrinth fractals defined by Cristea and Steinsky \cite{laby_4x4,laby_oigemoan, mixlaby} provide a
broad class of fractal labyrinths as described and used in physics and other
applied sciences, and which through their transparent construction method are
amenable to rigorous mathematical treatment. 
The results found for these mathematical objects have both potential and actual
applications in, and implications to, fields where their finite,``real''
counterparts occur, like, e.g., physics, material science, or life science.

\emph{Mixed labyrinth fractals} were introduced and studied in more recent work by Cristea and Steinsky \cite{mixlaby}. They are a generalisation of the self-similar labyrinth fractals introduced and studied by the same authors in previous work  \cite{laby_4x4, laby_oigemoan}.  In the case of mixed labyrinth fractals more that one pattern is used in order to construct the set,  as described in Section \ref{sec:Definition}. It has been proven \cite{mixlaby} that, when passing from the self-similar case to the generalised case of the mixed labyrinth sets and mixed labyrinth fractals, several of the topological properties are preserved: the mixed labyrinth fractals are dendrites in the unit square, too, that have exactly one exit on each side of  the unit square. In the self-similar case it was shown that special patterns, called blocked patterns, generate fractals that are dendrites with the property that the arc between any two points in the fractal has infinite length. 
\par
In the present article we show that in the case of mixed labyrinth fractals the situation is much more complex: on the one hand, one can find sequences of blocked labyrinth patterns that generate labyrinth fractals where the arc between any two points in the fractal is finite, and on the other hand one can find sequences of blocked labyrinth patterns whose resulting labyrinth fractal has the property that the arc between any two points of the fractal has infinite length. Moreover, we give an example for the construction of mixed labyrinth fractals where some arcs in the fractal have finite length and others have infinite length, analogous to the case when 
self-similar labyrinth fractals are generated by a pattern that is horizontally but not vertically, or vertically but not horizontally blocked  (see, e.g., \cite{laby_4x4}). Finally, we state a conjecture on lengths of arcs in mixed labyrith fractals, for future research.
\par
The results in this article provide ideas and modalities for constructing such fractal dendrites with desired properties regarding the lengths of arcs beween points in the fractal, that could serve as models, e.g., in the context of particle transport, nanostructures, image processing. We remark here that although there are several well known examples of continuous curves with infinite length, like the Peano curve \cite{Peano}, the Hilbert \cite{Hilbert} or the von Koch curve\cite{vonKoch1904, vonKoch1906}, not all of them have the property that the arc between any two points of the curve has infinite length, as in the case of  the arcs in some of the labyrinth fractals. Moreover, we note that random Koch curves, i.e. objects that are related, e.g., to arcs between certain points (exit points) in labyrinth fractals, are studied with respect to random walks by theoretical physicists in connection with diffusion processes, e.g. \cite{SeegerHoffmannEssex2009_randomKoch}. In this context we also mention diffusion processes of water in biological tissues. There are many more available examples that support the idea that labyrinth fractals, whether mixed or self-similar, are mathematical objects worth understanding with respect to their topological and geometrical properties, with benefits both in mathematics and in other sciences.

\section{Labyrinth fractals}\label{sec:Definition}

One way to construct labyrinth fractals is with the help of \emph{labyrinth patterns}.
Let $x,y,q\in [0,1]$ such that $Q=[x,x+q]\times [y,y+q]\subseteq [0,1]\times [0,1]$. 
For any point $(z_x,z_y)\in[0,1]\times [0,1]$ we define the function
$P_Q(z_x,z_y)=(q z_x+x,q z_y+y)$.

For any integer $m\ge 1$ let $S_{i,j,m}=\{(x,y)\mid \frac{i}{m}\le x \le \frac{i+1}{m} \mbox{ and } \frac{j}{m}\le y \le \frac{j+1}{m} \}$ and  
${\cal S}_m=\{S_{i,j,m}\mid 0\le i\le m-1 \mbox{ and } 0\le j\le m-1 \}$. 

Any nonempty ${\cal A} \subseteq {\cal S}_m$ is called an $m$-\emph{pattern} and $m$ its \emph{width}. Let $\{{\cal A}_k\}_{k=1}^{\infty}$ 
be a sequence of non-empty patterns and $\{m_k\}_{k=1}^{\infty}$ be the corresponding 
\emph{width-sequence}, i.e., for all $k\ge 1$ we have 
${\cal A}_k\subseteq {\cal S}_{m_k}$. 
We let $m(n)=\prod_{k=1}^n m_k$, for all $n \ge 1$. 
Let ${\cal W}_1={\cal A}_{1}$, we call ${\cal W}_1$ the 
\emph{set of white squares of level $1$}, and 
define ${\cal B}_1={\cal S}_{m_1} \setminus {\cal W}_1$ 
as the \emph{set of black squares of level $1$}. 
For $n\ge 2$ the \emph{set of white squares of level $n$} is defined as
\begin{equation} \label{eq:W_n}
{\cal W}_n=\bigcup_{W\in {\cal A}_{n}, W_{n-1}\in {\cal W}_{n-1}}\{ P_{W_{n-1}}(W)\}.
\end{equation}

 \noindent We remark that ${\cal W}_n\subset {\cal S}_{m(n)}$, and we define the \emph{set of black squares of level $n$} by ${\cal B}_n={\cal S}_{m(n)} \setminus {\cal W}_n$. For $n\ge 1$, we define $L_n=\bigcup_{W\in {\cal W}_n} W$. 
Thus, $\{L_n\}_{n=1}^{\infty}$ is a monotonically decreasing sequence of compact sets, and $L_{\infty}=\bigcap_{n=1}^{\infty}L_n$  is the \emph{limit set defined by the sequence of patterns 
$\{{\cal A}_k\}_{k=1}^{\infty}.$ }

Figures~\ref{fig:A1A2A3}, \ref{fig:W2}, and \ref{fig:pre_dendrite_general} show examples of labyrinth patterns and illustrate the first three steps of the 
construction of a mixed labyrinth set. 

We define, for ${\cal A}\subseteq {{\cal S}_m}$, $\mathcal{G}({\cal A})\equiv (\mathcal{V}(\mathcal{G}({\cal A})),\mathcal{E}(\mathcal{G}({\cal A})))$ to be \emph{the graph of ${\cal A}$}, i.e., the graph whose vertices $\mathcal{V}(\mathcal{G}({\cal A}))$ are the white squares in ${\cal A}$, i.e.,  $\mathcal{V}(\mathcal{G}({\cal A}))={\cal A}$ and whose edges $\mathcal{E}(\mathcal{G}({\cal A}))$ are the unordered pairs of white squares, that have a common side. The \emph{top row} in ${\cal A}$ is the set of all white squares in $\{S_{i,m-1,m}\mid 0\le i\le m^n-1 \}$. The bottom row, left column, and right column in ${\cal A}$  are defined analogously. A \emph{top exit} in ${\cal A}$ is a white square in the top row, such that there is a white square in the same column in the bottom row. A \emph{bottom exit} in ${\cal A}$  is defined analogously. A \emph{left exit} in ${\cal A}$ is a white square in the left column, such that there is a white square in the same row in the right column. A \emph{right exit} in ${\cal A}$ is defined analogously. 
One can of course define the above notions in the special case ${\cal A}={\cal W}_{n}$. In this case the top row (in ${\cal W}_{n}$) is called the 
\emph{top row of level} $n$. The \emph{bottom row, left column, and right column of level} $n$ are defined analogously.

\begin{figure}[hhhh]
\begin{center}
\includegraphics[width=0.6\textwidth]{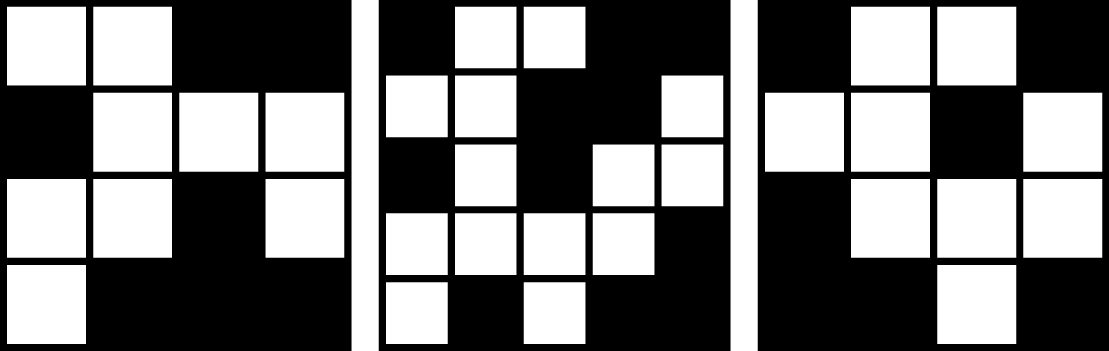}
\caption{Three labyrinth patterns, ${\cal A}_1$ (a $4$-pattern), ${\cal A}_2$ (a $5$-pattern), and ${\cal A}_3$ (a $4$-pattern)}
\label{fig:A1A2A3}
\end{center}
\end{figure}

\begin{figure}[hhhh]\label{W2}
\begin{center}
\includegraphics[width=0.3\textwidth]{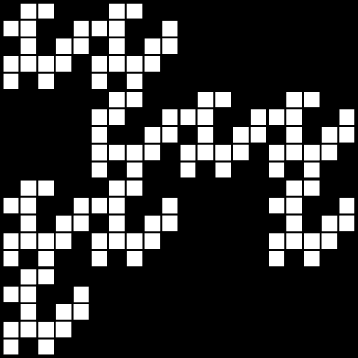}
\caption{The set ${\cal W}_2$, constructed based on the above patterns ${\cal A}_1$ and ${\cal A}_2$, that
can also be viewed as a $20$-pattern} \label{fig:W2}
\end{center}
\end{figure}

\begin{figure}[hhhh]
\begin{center}
\includegraphics[width=0.5\textwidth]{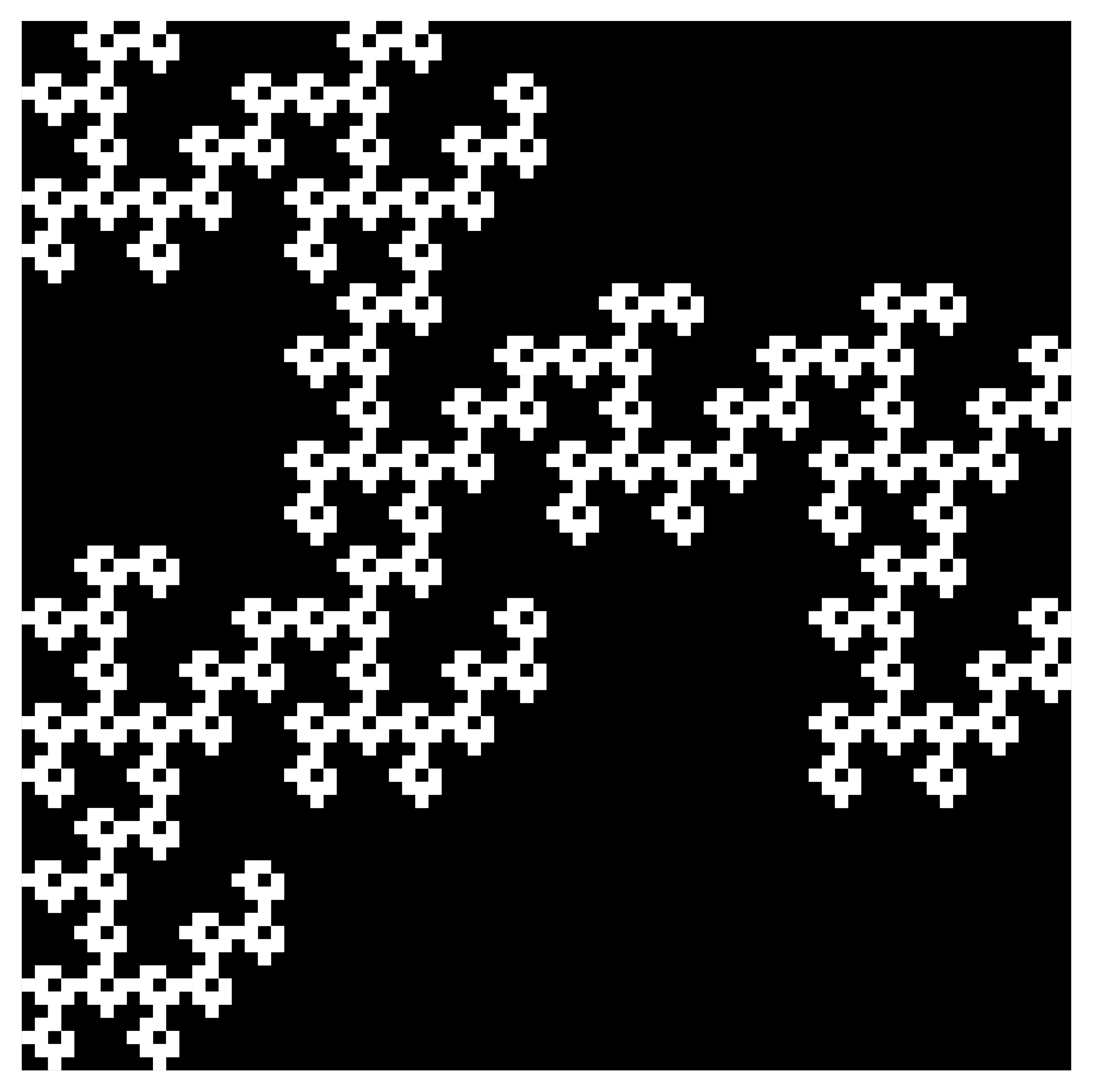}
\caption{A prefractal of the mixed labyrinth fractal defined by a sequence $\{{\cal A}_k\}$ where the first three patterns are ${\cal A}_1, {\cal A}_2, {\cal A}_3$, respectively, shown in Figure \ref{fig:A1A2A3}}  
\label{fig:pre_dendrite_general}
\end{center}
\end{figure}

A non-empty $m$-pattern ${\cal A} \subseteq {{\cal S}_m}$, $m \ge 3$ is called a $m\times m$-\emph{labyrinth pattern} (in short, \emph{labyrinth pattern}) if  ${\cal A}$ satisfies Property~\ref{prop1}, Property~\ref{prop2}, and Property~\ref{prop3}.
\begin{property}\label{prop1}
$\mathcal{G}({\cal A})$ is a tree.
\end{property}
\begin{property}\label{prop2}
Exactly one top exit in ${\cal A}$ 
lies in the top row, exactly one bottom exit lies in the bottom row, exactly one left exit lies in the left column, and exactly one right exit lies in the right column. 
\end{property}
\begin{property}\label{prop3}
If there is a white square in ${\cal A}$ at a corner of ${\cal A}$, then there is no white square in ${\cal A}$ at the diagonally opposite corner of ${\cal A}$. 
\end{property}

Let $\{{\cal A}_k\}_{k=1}^{\infty}$ be a sequence of non-empty patterns,  with $m_k\ge 3$, $n\ge 1$ and ${\cal W}_n$ the corresponding set of white squares of level $n$. We call ${\cal W}_{n}$ an $m(n)\times m(n)$-\emph{mixed labyrinth set} (in short, \emph{labyrinth set}), if ${\cal A} ={\cal W}_{n}$ satisfies Property~\ref{prop1}, Property~\ref{prop2}, and Property~\ref{prop3}.
It was shown \cite{mixlaby} that if all patterns in the sequence $\{{\cal A}_k\}_{k=1}^{\infty}$ are labyrinth patterns, then ${\cal W}_n$ is a labyrinth set, for any $n\ge 1$.
The limit set $L_{\infty}$ defined by a sequence $\{{\cal A}_k\}_{k=1}^{\infty}$ of labyrinth patterns is called \emph{mixed labyrinth fractal}. 

One can immediately see that in the special case when all patterns in the sequence $\{{\cal A}_k\}_{k=1}^{\infty}$ are identical, $L_{\infty}$ is a self-similar labyrinth fractal, as defined in \cite{laby_4x4, laby_oigemoan}.

In the following we introduce some more notation. For $n\ge 1$ and $W_1,W_2 \in \mathcal{V}(\mathcal{G}({\cal W}_{n}))$ we denote by $p_n(W_1,W_2)$ the path in $\mathcal{G}({\cal W}_{n})$ that connects $W_1$ and $W_2$.
A path in $\mathcal{G}({\cal W}_{n})$ is called $\A$\emph{-path} if it leads from the top to 
the bottom exit of $W_n$. 
The $\B,\C,\D,\E$, and $\F$\emph{-paths} lead from left to right, top to right, right to bottom, bottom to left, and left to 
top exit, respectively. 

Within a path in $\mathcal{G}({\cal W}_{n})$
 each white square in the path is denoted
according to its neighbours within the path: 
if it has a top and a bottom neighbour it is called $\A$-\emph{square}  
(with respect to the path), and
it is called $\B,\C,\D,\E$, and $\F$-\emph{square} if its neighbours are at left-right,
top-right, 
bottom-right, left-bottom, and left-top, respectively. 
If the considered square is an exit, it is supposed to have a neighbour outside the 
side of the exit. A bottom exit, e.g., is supposed to have a neighbour below, outside the bottom, additionally to its neighbour that lies inside the unit square.

For more details on labyrinth sets and mixed labyrinth fractals and for results on topological properties of mixed labyrinth fractals we refer to the paper \cite{mixlaby}.


\section{Existing results on arcs in mixed labyrinth fractals}\label{sec:old}

In this section we list some of the results obtained for  mixed labyrinth fractals \cite{mixlaby} that are useful in the context of this paper. We use the notation introduced in the previous section.

\begin{lemma}\label{lemma:Construction}(Arc Construction) Let $a,b\in L_{\infty}$, where $a\neq b$. For all $n \ge 1$, there are $W_n(a),W_n(b)\in V(\mathcal{G}({\cal W}_{n}))$ such that 
\begin{itemize}
\item[(a)]$W_1(a)\supseteq W_2(a)\supseteq\ldots$,
\item[(b)]$W_1(b)\supseteq W_2(b)\supseteq\ldots$,
\item[(c)]$\{a\}=\bigcap_{n=1}^{\infty}W_n(a)$,
\item[(d)]$\{b\}=\bigcap_{n=1}^{\infty}W_n(b)$.
\item[(e)]The set $\bigcap_{n=1}^{\infty}\left(\bigcup_{W\in p_n(W_n(a),W_n(b))} W\right)$ is an arc between $a$ and $b$. 
\end{itemize}
\end{lemma}
We recall from \cite{laby_4x4} that the squares $W_n(a)$, $W_n(b)$, $n\ge 1$ in the above lemma are chosen in the following way: let $W(a)$ be the set of all white squares in $\bigcup_{n=1}^{\infty} {\cal W}_n $ that contain $a$. Let $W_1(a)$ be a white square in ${\cal G} ({\cal W}_1)$ that contains infinitely  many white squares of $W(a)$ as a subset. For $n\ge 2$, we define $W_n(a)$ as a white square in ${\cal G}({\cal W}_n)$, such that $W_n(a)\subseteq W_{n-1}(a) $, and $W_n(a)$ contains infinitely many squares of $W(a)$ as a subset. $W_n(b)$, for $n\ge 1 $, is defined in the analogous manner.

\begin{proposition}\label{lemma:m^n} Let $n,k\ge 1$, $\{W_1,\ldots,W_k\}$ be a (shortest) path in 
$\mathcal{G}({\cal W}_{n})$ between the exits $W_1$ and $W_k$,  $K_0=W_1 \cap \fr([0,1]\times[0,1])$, $K_k=W_k \cap \fr([0,1]\times[0,1])$, where $\fr(\cdot)$ denotes the boundary of a set, and $c$ be a curve in $L_n$ 
from a point of $K_0$ to a point of $K_k$. The length of $c$ is at least $(k-1)/(2\cdot m(n))$.
\end{proposition}
Let $T_n\in {\cal W}_{n}$ be the top exit of ${\cal W}_{n}$, for $n\ge 1$. The \emph{top exit of} 
$L_{\infty}$ is $\bigcap_{n=1}^{\infty}T_n$. The other exits of $L_{\infty}$ are defined analogously. We note that 
Property~\ref{prop2} yields that $(x,1),(x,0)\in L_{\infty}$ if and only if $(x,1)$ is the top exit of $L_{\infty}$ and $(x,0)$ 
is the bottom exit of $L_{\infty}$. For the left and the right exit the analogous statement holds. 
Let $n\ge 1$, $W\in {\cal W}_{n}$, and $t$ be the intersection of $L_{\infty}$ with the top edge of $W$. 
Then we call $t$ the \emph{top exit} of $W$. Analogously we define the \emph{bottom exit}, the\emph{ left exit} and
the \emph{right exit} of $W$.
 We note that the uniqueness of each of these four exits is provided by the uniqueness of the four exits of a 
mixed labyrinth fractal and by the fact that each 
such set of the form 
$L_{\infty} \cap W$, where $W\in {\cal W}_{n}$, is a mixed labyrinth fractal scaled by the factor $m(n)$.
We note that we have now defined exits for 
three different types of objects, i.e., for ${\cal W}_{n}$ (and ${\cal A}_{k}$), for $L_{\infty}$, 
and for squares in ${\cal W}_{n}$.
\begin{proposition}\label{lemma:ArcSimilarity1} Let $e_1,e_2$ be two exits in $L_{\infty}$, and $W_n(e_1), 
W_n(e_2)$ be the exits in $\mathcal{G}({\cal W}_{n})$ of the same type as $e_1$ and $e_2$, respectively, 
for some $n\ge 1.$
If $a$ is the arc that connects $e_1$ and $e_2$ in $L_{\infty}$, $p$ is the path in 
$\mathcal{G}({\cal W}_{n})$ from $W_n(e_1)$ to $W_n(e_2)$, and $W\in {\cal W}_{n}$ is a $\A$-square 
with respect to $p$, then $W\cap a$ is an arc in $L_{\infty}$ between the top and the bottom exit of $W$. 
If $W$ is another type of square, the corresponding analogous statement 
holds.
\end{proposition}

\noindent For the  corresponding results, in detail,  for self-similar fractals we refer to  \cite{laby_oigemoan}.

\section{Blocked labyrinth patterns, blocked labyrinth sets and a recent conjecture}
\label{sec:Blocked}

We recall that an $m\times m$-labyrinth pattern ${\cal A}$ is called \emph{horizontally blocked} if the row (of squares) 
from the left to the right exit contains at least one black square. It is called \emph{vertically blocked} if the 
column (of squares) from the top to the bottom exit contains at least one black square. Analogously we define for any 
$n \ge 1$ a horizontally or vertically blocked labyrinth set of level $n$.
As an example, the labyrinth patterns shown in Figure \ref{fig:A1A2A3} and \ref{fig:complement} 
are horizontally and vertically blocked, while those in Figure \ref{fig:counterexample_1} are not blocked.

\begin{figure}
\begin{center}
\includegraphics{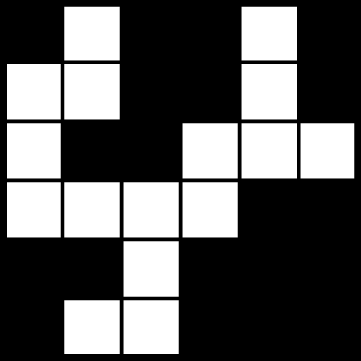}
\end{center}
\caption{A horizontally and vertically blocked ($6 \times 6$-labyrinth) pattern}\label{fig:complement}
\end{figure}

\begin{figure}[hhhh]
\begin{center}
\includegraphics[width=0.6\textwidth]{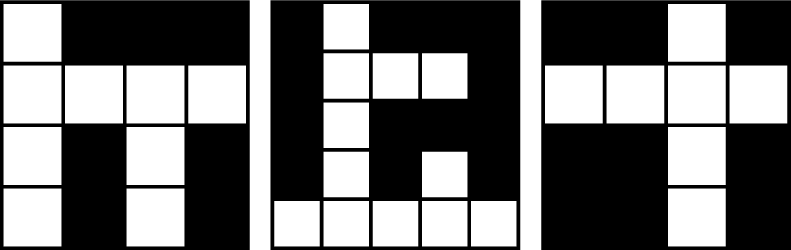}
\caption{Examples of labyrinth patterns, that are neither horizontally nor vertically blocked}
\label{fig:counterexample_1}
\end{center}
\end{figure}
In the self-similar case the following facts  were proven \cite[Theorem 3.18]{laby_oigemoan}:
\begin{theorem} Let $L_{\infty}$ be the (self-similar) labyrinth fractal generated by a horizontally and 
vertically blocked $m\times m$-labyrinth pattern. 
Between any two points in $L_{\infty}$ there is a unique arc ${a}$.
The length of ${a}$ is infinite.
The set of all points, at which no tangent to ${a}$ exists, 
is dense in ${a}$.
\end{theorem}

For the case of mixed labyrinth fractals, Cristea and Steinsky \cite{mixlaby} recently formulated the following conjecture. 

\begin{conjecture}\label{conj:main result}
 Let $\{{\cal A}_k\}_{k=1}^{\infty}$ be a sequence of both horizontally and vertically blocked labyrinth patterns,  
 $m_k\ge 4$. 
 For any two points in the limit set $L_{\infty}$ the length of the arc $a\subset L_{\infty}$ that connects them is infinite and the set of all points, where no tangent to $a$ exists, is dense in $a$.
 \end{conjecture}

In this article we solve the arc length problem posed by the above conjecture by showing that, depending on the choice of the both horizontally and vertically blocked labyrinth patterns in the sequence $\{{\cal A}_k\}_{k=1}^{\infty}$, both situations can occur: the arc between any points of the fractal has finite length, or the arc between any two points of the fractal has infinite length.

\begin{figure}[h!]
\begin{center}
\includegraphics[width=0.3\textwidth]{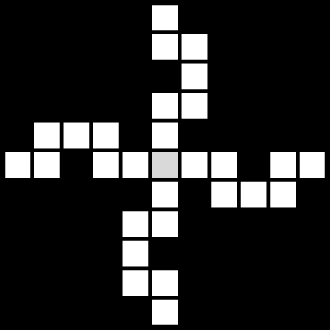}
\caption{An example: the special cross pattern ${\cal A}_1$ with $m_1=11$}
\label{fig:ex:A_1_finite_arc}
\end{center}
\end{figure}
\vspace{0.3cm}
\noindent {\bf Example 1.}
Let $\{{\cal A}_k\}_{k=1}^{\infty}$ be a sequence of (both horizontally and vertically) blocked labyrinth patterns, $m_k\ge 11$, with $m_k= 2 a_k+1$, $a_k\ge 5$.

We consider a sequence of patterns that are both vertically and horizontally blocked and have a 
``cross shape'' like ${\cal A}_1$ in Figure 
\ref{fig:ex:A_1_finite_arc}, i.e., the pattern looks like a ``cross" centered in the ``central square" of the 
pattern (here coloured in light grey) and each ``arm" of the cross is ``blocked" such that in order to get from the ``center" of the cross 
to any of the four exits of the pattern we have to go a detour around a black square that lies in 
the same row or column as the respective exit and the mentioned ``central square" of the pattern. More precisely, we position the four black squares between the central square 
and the exits of ${\cal A}_k$ in the columns (rows) $(a_k+1)/2$ and $(3a_k+3)/2$, if $a_k$ is odd, 
and in the columns (rows) 
$(a_k+2)/2$ and $(3a_k+2)/2$, 
if $a_k$ is even. We call these patterns 
\emph{special cross patterns}.
One can immediately see that the ``central square" of such a special cross pattern, where the four ``arms" of 
the cross meet, changes its type, depending on the path in ${\cal{G}}({\cal A}_1)$ that we consider between 
two exits of the pattern ${\cal A}_1$: in the $\A$-path in the pattern, it is a $\A$-square, and in the 
$y$-path, it is a $y$-square, for any $y\in \{\B,\C,\D,\E,\F \}$. 
\\[0.3cm]
We recall that the path matrix of a labyrinth set or a labyrinth pattern ${\cal A}$ is a 
$6\times 6$-matrix $M$ such that the element in  row $x$ and column $y$ is the number of 
$y$-squares in the $x$-path in $\mathcal{G}({\cal A})$. It was proven \cite[Proposition 1]{mixlaby} that, for any sequence of labyrinth patterns $\{{\cal A}_k\}_{k\ge 1}$ with corresponding sequence of path matrices  $\{M_k\}_{k\ge 1}$, for any integer $n\ge 1$ the matrix $M(n):= \prod_{k=1}^n M_k$ is the matrix of the mixed labyrinth set  ${\cal W}_n$ (of level $n$), i.e. the sum of the entries in any row of $M(n)$ gives the length of the path between two of the exits in $\mathcal{G}({\cal W}_n)$. For more details and properties of 
path matrices we refer to the papers \cite{ laby_4x4, laby_oigemoan, mixlaby}.\\[0.2cm]

With the help of Figures \ref{fig:ex:A_1_finite_arc} and \ref{fig:sequence_special_cross_patterns} one can easily check that for this special sequence of labyrinth patterns $\{ {\cal A}_k\}_{k \ge 1}$, the path matrix of the pattern ${\cal A}_k$ is 
\[
M_k=\left(
\begin{array}{rrllll}
2a_k-3 & 0 & 2 & 2 & 2 & 2\\
0 & 2a_k-3 & 2 & 2 & 2 & 2\\
a_k-2 & a_k-2 & 3 & 2 & 2 & 2\\
a_k-2 & a_k-2 & 2 & 3 & 2 & 2\\
a_k-2 & a_k-2 & 2 & 2 & 3 & 2\\
a_k-2 & a_k-2 & 2 & 2 & 2 & 3\\
\end{array}\right)
,
~~\mbox{for}~~ k\ge 1~~ \mbox{and}~~ m_k=2a_k+1.
\]

Thus the length of the path in $\mathcal{G}({\cal A}_k)$ between any two exits is in this case 
exactly $2a_k+5$.
Herefrom we obtain that the length of any path between two exits in ${\cal W}_n$ is then 
$\prod_{k=1}^n(2a_k+5)=\prod_{k=1}^n(m_k+4)$. 

\begin{figure}
\begin{center}
\includegraphics[width=0.18\textwidth]{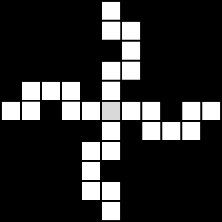}
\includegraphics[width=0.18\textwidth]{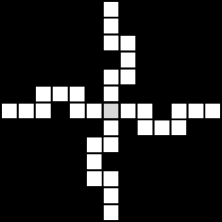}
\includegraphics[width=0.18\textwidth]{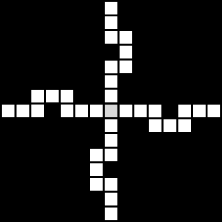}
\includegraphics[width=0.18\textwidth]{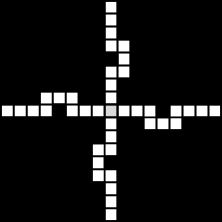}
\includegraphics[width=0.18\textwidth]{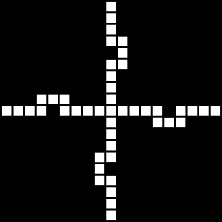}
\end{center}
\caption{Example: five consecutive elements of a sequence of special cross patterns, where $a_k=k+4$, here for $k=1,\dots,5$.}
\label{fig:sequence_special_cross_patterns}
\end{figure}

\begin{figure}[hhhh]
\begin{center}
\includegraphics[scale=0.5]{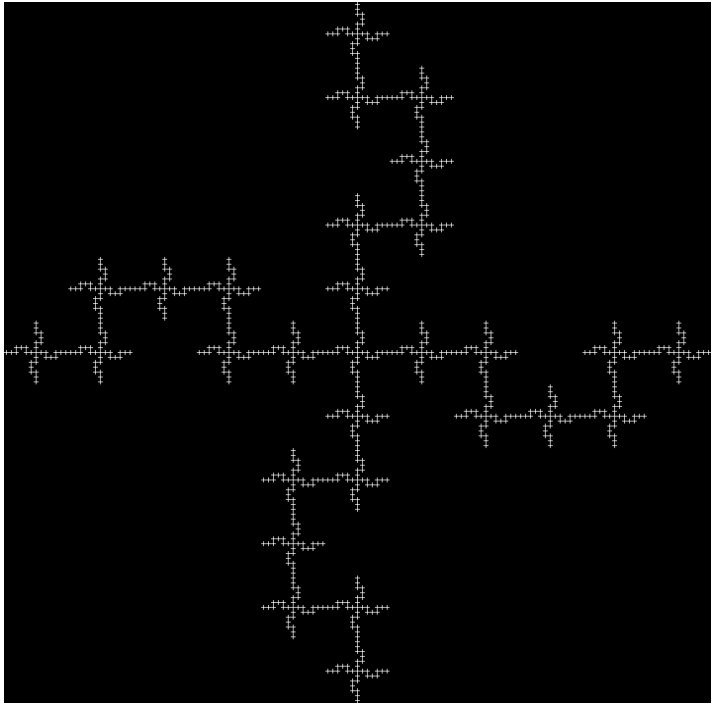}
\caption{A mixed labyrinth fractal of level 2 that is a prefractal of the mixed labyrinth fractal defined by a sequence $\{{\cal A}_k\}$ of special cross patterns as shown in Figure \ref{fig:sequence_special_cross_patterns}}
\label{fig:pre_cross5-6-7}
\end{center}
\end{figure}

As a next step, we introduce, for $n=1,2,\dots$, the curves $\gamma_n^q$, for 
$q \in {\cal E}=\{\A,\B,\C,\D,\E,\F  \}$. Here, $q  \in{\cal E}$ indicates which exits are connected by the path, 
e.g., if $q= \A$ then $\gamma_n^q$ is a curve that connects the top and the bottom exit of $L_{\infty}$, 
if $q = \B$ then $\gamma_n^q$ is a curve that connects the left and the right exit of $L_{\infty}$, and so on. 
In the sequel we define these simple curves. Let $n \ge 1$. For $q= \A$ we construct the curve 
$\gamma_n^q$ in the following way. Let $W\in {\cal W}_n$ be, e.g., a square of type $\A$ in the path of type $q$. 
Then, we define the restriction $ \gamma_n^q|W$ to be the vertical line segment that connects the 
midpoints of the top and of the bottom edge of $W$. We proceed analogously in the case when $W$ 
is a square of type $\B,\C,\D,\E,\F $, in each case $ \gamma_n^q|W$ is the union of two line segments
(both horizontal, or one horizontal and one vertical)
that both go through the center of $W$ and the midpoint of some edge of $W$, such that the sum of their lengths 
is $\frac{1}{m(n)}$. We immediately get the length of the curve $\gamma_n$, for $n\ge 1$:
\begin{equation}\label{length_gamma_n}
 \ell(\gamma_n)=\prod_{k=1}^n \frac{m_k+4}{m_k}=\prod_{k=1}^n\left( 1+\frac{4}{m_k}\right).
\end{equation}

\begin{figure}[h!]
\begin{center}
\includegraphics[width=0.7\textwidth]{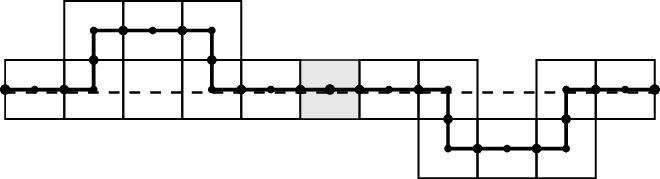}
\caption{This picture shows a fragment of $\gamma _n ^q$, and (dashed) $\gamma _{n-1} ^q$. Here $q \in {\cal E}$ indicates that the arc connects the left and right exit of a square $W \in {\cal W}_n$, in the picture the most left and most right dotted points.  }
\label{fig:consecutive_paths}
\end{center}
\end{figure}

Now we study the sequence $\{\ell(\gamma_n) \}_{n\ge 1}$. From \eqref{length_gamma_n} we easily see that
$\{\ell(\gamma_n) \}_{n\ge 1}$ is a strictly increasing sequence, thus $\lim_{n\to \infty} \ell(\gamma_n)=
\sup_{n=1,2,\dots} \ell(\gamma_n).$\\
\noindent
By basic mathematical analysis facts $\prod_{k=1}^n \left( 1+\frac{4}{m_k}\right)$ 
converges if and only if $ \sum_{k\ge 1} \frac{4}{m_k}$ converges, i.e., 
$ \sum_{k\ge 1} \frac{1}{m_k}<\infty$.
By taking, e.g., $a_k=5^k$, for $k=1,2,\dots$ we obtain 
$\sup_{n=1,2,\dots} \ell(\gamma_n)=\lim_{n \to \infty} \ell(\gamma_n)<\infty$.
\\[0.3cm        ]
\noindent
{\bf Remark.} One can verify, by using the definition of the Hausdorff distance $d_H$, that, for $q \in {\cal E}$,   
the arc $\gamma^q$ in $L_{\infty}$ that connects the two exits of $L_{\infty}$ indicated by $q$ satisfies 
$d_H(\gamma_n^q, \gamma^q)\to 0$, for $n\to \infty.$ 
Here we mean the Hausdorff distance between the images of the two curves, as sets in the Euclidean plane endowed with the Euclidan distance.

\begin{lemma}
\label{lemma:norminfinity_convergence} 
With the above notation, there are parametrisations ${\tilde{\gamma}}_n^q(t):[0,1]\to [0,1]^2$ 
and ${\tilde{\gamma}}^q(t):[0,1]\to [0,1]^2$   of ${\gamma}_n^q$ 
and ${\gamma}^q$, respectively, such that for all $q \in {\cal E}$, we have 
$||{\tilde{\gamma}}_n^q - {\tilde{\gamma}} ^q   ||_{\infty} \to 0$, for $n \to \infty$, where $|| \cdot||_{\infty}$ is the supremum norm.

\end{lemma}
\begin{proof}Sketch: the proof of Lemma \ref{lemma:norminfinity_convergence} is based on the fact that one can easily find parametrisations ${\tilde{\gamma}}_n^q(t):[0,1]\to [0,1]^2$ 
and ${\tilde{\gamma}}^q(t):[0,1]\to [0,1]^2$   of ${\gamma}_n^q$ and ${\gamma}^q$, respectively, such that
$||\tilde{\gamma}_n^q - \tilde{\gamma} ^q   ||_{\infty} \le \frac{3}{2}\cdot \frac{1}{m(n)}$, for $n=1,2, \dots$. See also Figure \ref{fig:consecutive_paths}.
\end{proof}

\begin{lemma}
 \label{lemma:unif_conv_implies_finite_length}
Let ${\tilde{\gamma}}_n:[0,1]\to [0,1]^2$ 
and ${\tilde{\gamma}}:[0,1]\to [0,1]^2$ be parametrisations of the planar curves ${\gamma}_n$ and ${\gamma}$, respectively, whose lengths we denote by $\ell({\gamma}_n)$ and $\ell({\gamma})$. If 
$||{\tilde{\gamma}}_n  -  {\tilde{\gamma}}||_{\infty} \to 0$, 
for ${n \to \infty}$, and $ \displaystyle \sup_{n}\ell({\gamma}_n) <\infty$, then $ \ell({\gamma})<\infty$. Moreover, in this case the following inequalities hold: $\displaystyle \liminf_{n \to \infty} \ell({\gamma}_n)  \le \ell({\gamma}) \le \limsup_{n \to \infty} \ell({\gamma}_n)$.
\end{lemma}

\begin{proof}
We give an indirect proof of the first assertion of the lemma, the second one we leave as an exercise, since for our purposes the first assertion is already enough. Assume $\ell({\gamma})=\infty$. Since  $ \displaystyle \sup_{n}\ell({\gamma}_n) <\infty$, we can choose a positive integer $N$ such that $\displaystyle N> \sup_{n}\ell({\gamma}_n)+1$. Then, by the definition of the length of a curve, there exist the real numbers $0=s_0<s_1<\dots < s_m=1$, $m \ge 1$, such that $\sum_{k=1}^m ||{\tilde{\gamma}}(s_k)- {\tilde{\gamma}}(s_{k-1}) ||>N$, where $||\cdot ||$ denotes the Euclidean norm in the plane. From the convergence hypothesis it follows that there exists an integer $n_0$ such that for every $n \ge n_0$ we have $||{\tilde{\gamma}}-{\tilde{\gamma}}_n ||_{\infty}<\frac{1}{2(m+1)}$, and thus $\displaystyle \max_{k=1,\dots,m} ||{\tilde{\gamma}}(s_k)-{\tilde{\gamma}}_n(s_k) ||<\frac{1}{2(m+1)}$. Moreover, by the definition of the length of a curve, $\ell(\gamma _n)\ge \sum_{k=1}^m ||{\tilde{\gamma}}_n(s_k)- {\tilde{\gamma}}_n(s_{k-1}) ||$.
Thus, we now easily obtain the following inequalities:

\begin{align*}
N &< \sum_{k=1}^m ||{\tilde{\gamma}}(s_k)- {\tilde{\gamma}}(s_{k-1}) || \\
& \le \sum_{k=1}^m \left(||{\tilde{\gamma}}(s_k)- {\tilde{\gamma}_n}(s_{k}) ||+||{\tilde{\gamma}}_n(s_k)- {\tilde{\gamma}}_n(s_{k-1}) ||+||{\tilde{\gamma}}_(s_{k-1})- {\tilde{\gamma}}(s_{k-1}) ||\right) \\
& \le 2 \sum_{k=0}^m ||{\tilde{\gamma}}(s_k)- {\tilde{\gamma}}_n(s_{k}) || + \sum_{k=1}^m ||{\tilde{\gamma}}_n(s_k)- {\tilde{\gamma}}_n(s_{k-1}) || \le 1 + \ell(\gamma_n), 
\end{align*}
which leads to a contradiction.
\end{proof}

Now, for an arbitrary $n\ge 1$, let us take $W \in {\cal W}_n$, $L_{\infty}| W:=L_{\infty}\cap W$ and consider any two of the exits $e_1$, $e_2$ of the square $W$ (as defined in the paper \cite{mixlaby}). Then the arc in $L_{\infty}| W$ that connects $e_1$ and $e_2$ is the scaled image of the arc between two exits (of the same types) of a labyrinth set $L'_{\infty}$ generated by the sequence of patterns $\{{\cal A'}_k\}_{k=1}^{\infty}$, where ${\cal A'}_k={\cal A}_{k+n}$, and the scaling factor is $m(n)$. Therefore, one can easily see that the arc beween any such exits of any square  $W \in {\cal W}_n$, for any $n\ge 1$, is finite. 
\\
Herefrom it then easily follows that if $x,y$ are points  that belong to the set  of  points in $L_{\infty}$ that consists of all centres and all exits of squares of $\cup_{n\ge 1} {\cal V}({\mathcal G}({\cal W}_n))$, then the length of the arc in $L_{\infty}$ that connects $x$ and $y$ is finite.

Let ${E}_n$ be the set of all points of $L_{\infty}$ that are exits of squares of level $n$, and  ${C}_n$ be the set of all points of $L_{\infty}$ that are centers of squares of level $n$.  For any two distinct points $x',y' \in L_{\infty}$, we introduce the notation $a (x',y')$ for the arc in $ L_{\infty}$ that connects the points $x'$ and $y'$. Let now $W\in {\cal V}({\mathcal G}({\cal W}_n))$, with $ n\ge 0$ (for $n=0$, $W$ is the unit square, otherwise it is a white square of level n, as defined before). Let $c$ be the center of $W,$ and $e$ one of its four exits.
Now, we want to show  that for any point $x \in (\inter W \cap L_{\infty})\setminus \cup_{k\ge n}(E_k \cup C_k)$, where $\inter$ denotes the interior of the set, $\ell(a(x,c))<\infty$ and $\ell(a(x,e))<\infty$. In the following we give a proof of the first inequality.
Therefore, we consider two cases.
 
First, we assume
 that $x$ is a point on one of the four ``main arms'' of $L_{\infty} \cap W$ (which is in fact the scaled image of the mixed labyrinth fractal defined by the sequence $\{{\cal A}_k\}_{k \ge n+1}$), i.e., $x$ lies on the arc in $L_{\infty}$ that connects the center of $W$ with one of its exits. In this case, it easily follows from the above results that the length of the arc in $L_{\infty}$ between $x$ and the center of $W$ has finite length (that is less than one half of the length of the arc between two exits in $W$).

 In the second case, we assume that $x$ does not lie on a ``main arm'' of $L_{\infty} \cap W$, i.e., $X$ lies on a ``branch'' of the dendrite, that originates at a point, say $c'$, with $c'\in \bigcup_{k\ge n+1}C_k,$ that lies on one of the four  ``main arms'' of  $L_{\infty} \cap W$ (that connects the center $c\in C_n$ of $W$ with one of its exits, say $e\in E_n$).
By the construction of the fractal and of arcs in the fractal (Lemma \ref{lemma:Construction}), there exists a point $ e'\in \bigcup_{k\ge n+1}E_k $ such that $x$ lies on the arc  $a(c',e')$ in $L_{\infty}$ that connects $c'$ and $e'$, which, due to the above considerations, has finite length. 
Since $ \ell(a(c,x))=  \ell(a(c,c'))+ \ell(a(c',x))\le \ell(a(c,c'))+ \ell(a(c',e'))= \ell(a(c,e'))<\infty $, it follows that $ \ell(a(c,x)) <\infty$.
We leave the proof of the inequality $\ell(a(x,e))<\infty$ to the reader as an exercise.

 Let now $x,y\in L_{\infty}$ be two distinct points, and let $W_n(x)$ and $W_n(y)$ be two squares in ${\cal W}_n$ such that $x\in W_n(x)$ and $y\in W_n(y)$. The squares $W_n(x)$, $n\ge 1$ are chosen in the following way: let $W(x)$ be the set of all white squares in $\bigcup_{n=1}^{\infty} {\cal W}_n $ that contain $x$. Let now $W_1(x)$ be a white square in ${\cal G} ({\cal W}_1)$ that contains infinitely  many white squares of $W(x)$ as a subset. Now we define, for $n\ge 2$, $W_n(x)$ as a white square in ${\cal G}({\cal W}_n)$, such that $W_n(x)\subseteq W_{n-1}(x) $, and $W_n(x)$ contains infinitely many squares of $W(x)$ as a subset. We define $W_n(y)$, for $n\ge 1,$ in the analogous manner.
Let $p_n$ be the path between $W_n(x)$ and $W_n(y)$ constructed as in Lemma \ref{lemma:Construction}. Since $x\ne y$, it follows that there exists an integer $k\ge 1$, such that $p_k$ consists of at least $3$ squares. Let then $W \in p_n$ be a square with $W \notin \{ W_n(x), W_n(y)\}$. By the construction of the arc $a$ in $L_{\infty}$ between $x$ and $y$ as described in Lemma \ref{lemma:Construction}, $a \cap W$ is an arc between two exits of W, and thus has finite length. Since $a(x,y)$ is the union of finitely many such arcs of finite length with the arcs $a(x,e_x)$ and $a(e_y,y)$, where $e_x$ ist one of the exits of $W_n(x)$ and $e_y$ is one of the exits of $W_n(y)$, namely $\{e_x\}=a(x,y)\cap \fr(W_n(x))$, and  $\{e_y\}=a(x,y)\cap \fr(W_n(y))$, it follows that $a(x,y)$ has infinite length. \\

The above example shows that one can find a sequence of patterns that generates a mixed labyrinth fractal with the property that the length of the arc that connects any two points in the fractal is finite. Moreover, one can see that for a labyrinth pattern that contains such a ``special cross" the length of the paths between the exits of the pattern does not change, it is the same as here, and thus the arc lengths in the fractal also remain finite, as in the above example.

Thus we have proven the following result.
\begin{proposition}
\label{prop:exist_patterns_finite_arcs}
There exist  sequences $\{{\cal A}_k\}_{k=1}^{\infty}$ of (both horizontally and vertically) blocked labyrinth patterns, such that the limit set $L_{\infty}$ has the property that
 for any two points in $L_{\infty}$ the length of the arc $a\subset L_{\infty}$ that connects them is finite.
\end{proposition}
Based on a theorem in the book of Tricot \cite[p.73, Chapter 7.1]{Tricot} regarding the existence of the tangent to a curve of finite length, we obtain the following stronger result:
\begin{theorem}
There exist  sequences $\{{\cal A}_k\}_{k=1}^{\infty}$ of (both horizontally and vertically) blocked labyrinth patterns, such that the limit set $L_{\infty}$ has the property that
 for any two points in $L_{\infty}$ the length of the arc $a\subset L_{\infty}$ that connects them is finite. For almost all points $x_0 \in a$ (with respect to the length) there exists the tangent at $x_0$ to the arc $a$.
\end{theorem}

\noindent
{\bf Remarks}
\begin{enumerate}
\item It is easy to see that such special cross patterns as shown in Figure 
\ref{fig:ex:A_1_finite_arc}, with such a ``detour'' on each of the four arms, exist only for a pattern with width $m \ge 11$. Moreover, one can easily check that for the above example, both the box-counting  and the Hausdorff dimension of the fractal is $\dim_B(L_{\infty})=\dim _H (L_{\infty})=1$ and also the box-counting dimension of any arc that connects a pair of exits in $L_{\infty}$ is $1$. The same holds for the arc between any two distinct points in the fractal. 
\item By the definition of a mixed labyrinth fractal, by the shape of special
cross patterns, and the arc construction given in Lemma
\ref{lemma:Construction},  one can check that the fractal is the countable
union of rectifiable $1$-sets. An example of such a countable collection of
rectifiable $1$-sets is as follows: take, for any level $n\ge 1$ of the
construction, the arcs in $L_{\infty}$ that connect the center of any square
$W \in \mathcal{G}({\cal W}_n)$ and any of the midpoints of its sides, i.e.,
any of the four exits of $W,$ as well as the four arcs in $L_{\infty}$ that
connect the center of the unit square with any of its midpoints (the four exits
of the mixed labyrinth fractal).  

\item 
In Example 1 we could also take, e.g., cross patterns with $a_k=2^k$, for $k \ge 1$, and consider the first two patterns in the sequence of generating patterns, to be just unblocked, symmetric cross patterns, with the width $m_k=2 a_k+1$, $k\in \{ 1,2\}$, and for $k\ge 3$ special cross patterns. Then, ${\cal W}_n$ is blocked  for all $n\ge 3$, and the resulting limit set $L_{\infty}$ would still have the property that the arc between any two points in the fractal has finite length. 
\end{enumerate}
In the following example we show that one can use blocked cross patterns like the one shown in Figure \ref{fig:ex:A_1_finite_arc} in order to construct mixed labyrinth fractals with the property that the arc between any two points in the fractal has infinite length.
\\\\
\noindent {\bf Example 2.} Let $\{{\cal A}_k\}_{k=1}^{\infty}$ be a sequence of special cross patterns like those occurring in Example 1,  $r \ge 2$ be an arbitrarily fixed integer, and let  
$\{{\cal A}'_i,~i=1,\dots,r\} \subset \{{\cal A}_k\}_{k\ge 3} $ be a finite set of blocked labyrinth patterns among those in the above infinite  sequence, where $m'_i=2a'_i+1$ denotes the width of the pattern ${\cal A}'_i$, and $l'_i=2a'_i+5$ is the lenght of the path between any two exits in  ${\cal G}({\cal A}'_i)$,  for $i=1,\dots,r$. We define a new sequence of labyrinth patterns  $\{{\cal A}^*_j\}_{j\ge1}$, e.g., in the following way: ${\cal A}^*_j \in \{  {\cal A}'_i, ~i = 1,\dots,r\}$. Let $L_{\infty}$ be the mixed labyrinth fractal generated by the sequence of patterns $\{{\cal A}^*_j\}_{j\ge1}$, and let $a^*:=\max \{a'_i,~i=1,\dots ,r\} $, $m^*:=2a^*+1$, and $l^*:=2a^*+5$. 
Since 

\[
\prod_{{i=1,\dots,r, \atop {k_1+\dots+k_r=n,} }\atop {k_i\ge 0}}\left(\frac{l'_i}{m'_i}\right)^{k_i}>\left(\frac{l^*}{m^*}\right)^n=\left(1+\frac{4}{m^*} \right)^n\to \infty, ~~\text{for}~n\to \infty,
\]
it follows from Lemma \ref{lemma:Construction} that the length of the arc between any two exits in  $L_{\infty}$ is infinite.
By using arguments analogous to those in Example 1, one can show that the infinite length of the arc between any two exits of $L_{\infty}$ implies that the arc between any two points in the fractal is infinite, as in the case of self-similar labyrinth fractals generated by both horizontally and vertically blocked patterns \cite{laby_4x4, laby_oigemoan}.

Thus we have proven the following
\begin{proposition}
There exist  sequences $\{{\cal A}_k\}_{k=1}^{\infty}$ of (both horizontally and vertically) blocked labyrinth patterns, such that the limit set $L_{\infty}$ has the property that
 for any two points in $L_{\infty}$ the length of the arc $a\subset L_{\infty}$ that connects them is infinite.
\end{proposition}

\begin{figure}[h!]
\begin{center}
\includegraphics[width=0.3\textwidth]{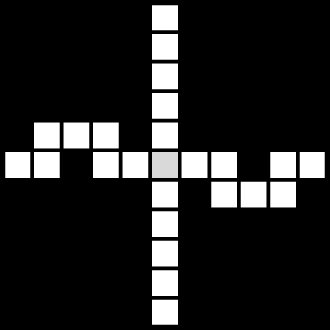}
\caption{An example: a half-blocked cross pattern ${\cal A}_1$ with width $m_1=11$ that is horizontally blocked, but not vertically blocked}
\label{fig:ex:A_1_finite_and_infinite}
\end{center}
\end{figure}

\noindent
{\bf Example 3. } 
In Figure \ref{fig:ex:A_1_finite_and_infinite} we have a ``half-blocked'' labyrinth pattern with width $11$, that is horizontally, but not vertically blocked, we call such a pattern (of width $m=2a+1\ge 11$) a \emph{half-blocked cross pattern}, where either the horizontal or the vertical arms of the cross make a detour around a black square, positioned as in the case of the special cross patterns used in Example 1, and the ``central square'' of the pattern (where all ``cross arms'' meet) is positioned in column and row $a+1$.
Suppose we have a sequence  $\{{\cal A}_k\}_{k=1}^{\infty}$ of such patterns, with $m_k=2a_k+1,$ and $a_k\ge 5$, for $k\ge 1$. 
For any pattern ${\cal A}_k$, with width $m_k=2a_k+1$, of the above sequence, the path matrix is 

\[
M_k=\left(
\begin{array}{rrllll}
2a_k+1 & 0 & 0 & 0 & 0 & 0\\
0 & 2a_k-3 & 2 & 2 & 2 & 2\\
a_k & a_k-2 & 2 & 1 & 1 & 1\\
a_k & a_k-2 & 1 & 2 & 1 & 1\\
a_k & a_k-2 & 1 & 1 & 2 & 1\\
a_k & a_k-2 & 1 & 1 & 1 & 2\\
\end{array}\right).
\]
Thus, the lengths of the paths between exits in ${\cal G}({\cal A}_k)$ are: $\A _k=2a_k+1$, $\B _k=2a_k+5$, $\C _k=\D _k=\E _k=\F _k=2a_k+3$. 
 One can immediately check that the length of the arc between the top and bottom exit of the resulting labyrinth set $L_{\infty}$ is $1$, no matter how we chose the sequence $\{ a_k\}_{k\ge 1}$. Now, let us analyse the arc  between the left and the right exit in $L_{\infty}$, and denote by $\B(n)$ the path in ${\cal G}({\cal W}_n)$ between the left and right exit in ${\cal W}_n$. If we choose, e.g., $a_k=k+4$, for $k\ge 1$, then $\sum_{k=1}^{\infty}\frac{1}{m_k}=\infty$. 

From Proposition \ref{lemma:m^n} we have, with the notation used above: for $q\in \{ \B \}  $,  
\[ \ell(\gamma^q)\ge \frac{\prod_{k=1}^n(m_k+4)-1}{2 \prod_{k=1}^n m_k}=\frac12\prod_{k=1}^n\Big(1+\frac{4}{m_k}\Big)-\frac{1}{2\prod_{k=1}^n m_k}
.
\]
 Under the above assumptions, $\displaystyle \lim_{n\to \infty}\Big(\frac12\prod_{k=1}^n\big(1+\frac{4}{m_k}\big)-\frac{1}{2\prod_{k=1}^n m_k}\Big)=\infty,$ for $n\to \infty$, as one can immediately check. Herefrom one can easily infere that also the arcs in $L_{\infty}$ that connect the top or the bottom exit of $L_{\infty}$ with the left or the right exit of $L_{\infty}$ have, all four, infinite length.

Moreover, one can show, by using arguments analogous to those mentioned when proving Proposition \ref{prop:exist_patterns_finite_arcs}, that for any $W\in {\cal W}_n$ the arc $L_{\infty}$ that connects the left and the right exit of $W$ has infinite length. This also holds for
 the arc in $L_{\infty}$ that connects the top or the bottom exit of $W$ with the left or the right exit of $W$.

\begin{proposition}
There exist sequences of horizontally blocked and not vertically blocked labyrinth patterns, such that the resulting mixed labyrinth  fractal $L_{\infty}$ has the following properties:
\begin{enumerate}
\item The arc in the fractal that connects  the top and the bottom exit of $L_{\infty}$ has finite length (equal to $1$). The arc in the fractal that connects the  top and bottom exit of any square in ${\cal W}_n$ is a vertical segment with finite length. Any vertical line segment that is contained in the fractal has finite length.
\item The arc in the fractal that connects  the left and the right exit of $L_{\infty}$ has infinite length. The arc in the fractal that connects the left and right exit of any square in ${\cal W}_n$ has infinite length.
\item The arc in the fractal that connects the exit $e_1$ and the exit $e_2$ in $L_{\infty}$, where $e_1$ is either the top or the bottom exit, and $e_2$ either the left or the right exit, has infinite length, and the same holds for the arcs between these pairs of exits in any square in ${\cal W}_n$.
\end{enumerate} 
\end{proposition}

The corresponding analogous statements regarding the existence of sequences of vertically and not horizontally blocked labyrinth patterns hold, such that $L_{\infty}$ has the corresponding analogous properties.
\par
For the above sequence of half-blocked cross patterns the length of the arc in $L_{\infty}$ that connects two arbitrarily chosen distinct points $x,y \in L_{\infty}$ has finite length if and only if it is contained in a vertical line segment which is, itself, a subset of $L_{\infty}$.

\vspace{2mm}
\par \noindent
{\bf Some final remarks.} 
In the case of mixed, i.e., not self-similar labyrinth fractals not just the shape of the patterns but also their width plays an essential role as a parameter that influences the lengths of arcs between exits or between any points in the fractal.
\\\\
{\bf Conjecture:}  A sequence of both horizontally and vertically blocked labyrinth patterns with the property that the sequence of widths $\{m_k\}_{k\ge 1}$ is bounded, generates a mixed labyrinth fractal with the property that for any $x,y \in L_{\infty}$ the length of the arc in the fractal that connects $x$ and $y$ is infinite.
\\\\{\bf Acknowledgement.} The authors thank Bertran Steinsky for valuable remarks on the manuscript.  We thank the referee for helpful comments.

\end{document}